\documentclass[12pt,a4paper]{amsart}
\usepackage{amsmath,amsthm,amsfonts,amssymb,mathrsfs}

\newcommand\R{\mathbb{R}}

\newcommand\N{\mathbb{N}}

\newcommand\G{\mathcal{G}}
\newcommand\U{\mathcal{U}}

\newtheorem{theorem}{Theorem}

\newtheorem{lemma}[theorem]{Lemma}
\newtheorem{prop}[theorem]{Proposition}

\theoremstyle{remark}
\newtheorem{rem}[theorem]{Remark}

\def\norm#1#2{\|#1\|_{L^#2}}%090207
\usepackage{delarray}
\usepackage{enumerate}

\numberwithin{equation}{section}
\numberwithin{theorem}{section}
\begin{document}

%%%%%%%%%%%%%%%%%%%%%%%%%%%%%
%%%%%%%%%%%%%%%%%%%%%%%%%%%%
\title[one-dimensional model of chemotaxis]{Spikes and diffusion waves\\  in one-dimensional model of chemotaxis}
\

\author{Grzegorz Karch}
\address{
 Instytut Matematyczny, Uniwersytet Wroc\l awski,
 pl. Grunwaldzki 2/4, 50-384 Wroc\-\l aw, POLAND}
\email{grzegorz.karch@math.uni.wroc.pl}
\urladdr{http://www.math.uni.wroc.pl/~karch}

\author{Kanako Suzuki}
\address{
Institute for International Advanced Interdisciplinary Research,
Tohoku University,
6-3 Aramaki-aza-Aoba, Aoba-ku, Sendai 980-8578, JAPAN}
\email{kasuzu-is@m.tains.tohoku.ac.jp}

\date{\today}

\thanks{
This work was partially supported 
by the Polish Ministry of Science grant N201 022 32/0902, the
Japan-Poland Research Cooperative Program (2008-2009), MEXT the Grant-in-Aid
for Young Scientists (B) 20740087, and 
the Foundation for Polish Science operated within the
Innovative Economy Operational Programme 2007-2013 funded by European
Regional Development Fund (Ph.D. Programme: Mathematical
Methods in Natural Sciences).}

\begin{abstract} 
We consider the one-dimensional initial value problem for the
 viscous transport equation with nonlocal velocity
$u_t = u_{xx} - \left(u (K^\prime \ast u)\right)_{x}$ 
with a given kernel $K'\in L^1(\R)$.
We show the existence of global-in-time nonnegative solutions and
we study their large time asymptotics.
Depending on $K'$, we obtain either linear diffusion waves ({\it i.e.}~the fundamental solution of the heat equation)
or nonlinear diffusion waves (the fundamental solution of the viscous Burgers  equation) in  asymptotic 
expansions of solutions as $t\to\infty$. Moreover, for certain aggregation kernels, we show  a concentration
of solution on an initial time interval, which resemble a phenomenon of the spike creation, typical  in chemotaxis models. 
\end{abstract}

%%%%%%%%%%%%%%%%%%%%%%%%%%%%%%%%%%%%%%%%%%%%%%%%%%%%%%%%%%%%%%%%%%%%%
%%%%%%%%%%%%%%%%%%%%%%%%%%%%%%%%%%%%%%%%%%%%%%%%%%%%%%%%%%%%%%%%%%%%%
\bigskip

%%%%%%%%%%%%%%%%%%%%%%%%%%%%%%%%%%%%%%%%%%%%%%%%%%%%%%%%%%%%%%%%%%%%%

\maketitle

%%%%%%%%%%%%%%%%%%%%%%%%%%%%%%%%%%%%%%%%%%%%%%%%%%%%%%%%%%%%
%%%%%%%%%%%%%%%%%%%%%%%%%%%%%%%%%%%%%%%%%%%%%%%%%%%%%%%%%%%%
%%%%%%%%%%%%%%%%%%%%%%%%%%%%%%%%%%%%%%%%%%%%%%%%%%%%%%%%%%%%

\section{Introduction}

%%%090902
%%%%%%%
In this work, we study the large time behavior  of solutions to the one-dimensional initial value problem
\begin{align}
&u_t = u_{xx} - \left(u (K^\prime \ast u)\right)_{x} \quad \text{for}\ x
 \in \R,\ t > 0, \label{intro-eq1}\\
&u(x, 0) = u_0 (x)\quad \text{for}\ x \in \R, \label{intro-eq2}
\end{align}
where the {\it aggregation} kernel $K^\prime \in L^1 (\R)$ is a given
function (the symbol ``$*$'' denotes the convolution with respect to the variable
$x$) %
and the initial datum $u_0 \in L^1 (\R)$ is nonnegative. %
Such models have been used to describe  a collective motion and aggregation phenomena in biology and
mechanics of continuous media.
In this case, the unknown function $u = u(x, t)\geq 0$ 
is either the population density of  a species or the density of particles in a granular media.
The kernel $K'$ in  \eqref{intro-eq1} can be understood as the
derivative of a certain function $K$, that is, $K^\prime$ stands for $dK/dx$.  We use 
this notation to emphasize that cell interaction described by equation \eqref{intro-eq1} takes place by
means of a potential $K$ (see \cite{BV05} for derivation of such equations).

Equation \eqref{intro-eq1} and its generalizations, considered either in the whole space or in a bounded domain, 
have been studied in several recent works.
First, notice that in  the particular case of  $K (x) = e^{-|x|}/2$, equation
\eqref{intro-eq1} corresponds to the  one-dimensional
parabolic-elliptic system of chemotaxis
\begin{equation}
u_t = u_{xx} - (u v_x)_x ,  \quad  -v_{xx} +v= u , \qquad x\in\R, \; t>0. \label{chemo}
\end{equation}
Indeed, since 
$K(x) = e^{-|x|}/2$ is the fundamental solution of the operator
$-\partial_x^2 + {\rm Id}$, one can rewrite the second equation in \eqref{chemo} as $v=K*u$. Now, 
it suffices to substitute this formula into the first equation in \eqref{chemo} to obtain \eqref{intro-eq1}.
We refer the reader to the 
 works  \cite{BK10, BKL09, H07, BDP, Kozono-Sugiyama,Murray} (this list is far from being complete)
for mathematical results  and for additional references 
 on systems modeling chemotaxis.
 
 Next, one should mention 
  the inviscid aggregation equation
$u_t+\nabla\cdot (u\nabla K*u)=0$, describing the evolution of a cell density, which was derived as a macroscopic 
equation from the so-called ``individual cell-based model'' \cite{S00, BV05}. Here, the reader is referred to 
 \cite{ BCL1, BL1, BV05, BV06, Lau1}
 for recent results on the existence and the blowup of solutions to the initial value problem for the
inviscid aggregation equation.

To handle diffusion phenomena, equations describing aggregation are supplemented with additional terms. One possible approach 
is to add a nonlinear term modeling a degenerate diffusion as in the porous medium equation, see {\it e.g.}~\cite{V02}. 
Results and other references on
 the chemotaxis model with  a degenerate diffusion can be found in \cite{S09, BCL09}
and on more general aggregation equations in \cite{IN87, TBL06, LZ10}.

In several cases, the mechanism of spreading out of organisms resembles a L\'evy flight, hence, the anomalous diffusion is
better modeled by  nonlocal pseudodifferential operators. Recent works \cite{BK10,BKL09,E06,Li-Rod1} contain several mathematical results 
on a chemotaxis system and on an aggregation equation with either the fractional Laplacian or a more general
L\'evy operator.

In our recent work \cite{KS09}, 
we have studied the multidimensional version of problem \eqref{intro-eq1}-\eqref{intro-eq2}
and we have answered questions how singularities of the gradient of the aggregation kernel $K$ 
influence on the existence and the nonexistence of global-in-time solutions.
In this paper, we complete those results in the one-dimensional case by
proving the existence of global-in-time solutions for every $K^\prime
\in L^1 (\R)$, %
and by studying their large time asymptotics. %
We show that asymptotic profiles, as  $t \to \infty$, of solutions to \eqref{intro-eq1}-\eqref{intro-eq2} 
with general $K^\prime \in L^1 (\R)$
are given either by the fundamental solution of the linear heat equation or by %
self-similar solutions of the viscous Burgers equation. Moreover, under another set of assumptions (which are
satisfied {\it e.g.}~by the kernel $K(x) = e^{-|x|}$), 
we prove a certain concentration property solutions to  \eqref{intro-eq1}-\eqref{intro-eq2} 
which can be observed on an initial time interval.

\subsection*{Notation}
In this work, the usual norm of the Lebesgue space 
$L^p (\R)$ is denoted by $\|\cdot\|_p$ for any $p \in [1,\infty]$ and $W^{k,p}(\R)$ is the corresponding Sobolev space. 
$C^\infty_c (\R)$ denotes the set of smooth and compactly
 supported functions. 
The constants (always independent
of $x$ and $t $) will be
denoted by the
same letter $C$, even if they may vary from line to line.
Sometimes, we write,  {\it e.g.},  $C=C(\alpha,\beta,\gamma, ...)$ when we want to
emphasize
the dependence of $C$ on parameters~$\alpha,\beta,\gamma, ...$.

%%%%%%%%%%%%%%%%%%%%%%%%%%%%%%%%%%%%%%%%%%%%%%%%%%%%%%%%%%%%
%%%%%%%%%%%%%%%%%%%%%%%%%%%%%%%%%%%%%%%%%%%%%%%%%%%%%%%%%%%%
%%%%%%%%%%%%%%%%%%%%%%%%%%%%%%%%%%%%%%%%%%%%%%%%%%%%%%%%%%%%
\section{Results and comments}

We begin our study of  properties of solutions to the initial value problem \eqref{intro-eq1}--\eqref{intro-eq2}
by showing  that it
  has  a unique and global-in-time solution
for a large class of initial conditions and aggregation kernels. 
The results on the global-in-time existence and regularity of solutions from  the following theorem
are more-or-less standard and we state them for the completeness of the exposition. 
They are also not surprising because it is well-known that solutions to the one-dimensional  Patlak-Keller-Segel  model of
chemotaxis do not blow up in finite time (see {\it e.g.}~\cite{OY01,HP04,E06}).
On the other hand, the fact that all solutions to problem 
\eqref{intro-eq1}--\eqref{intro-eq2}  are uniformly bounded in time as the $L^p$-valued functions 
(see Proposition~\ref{sec-ex-prop1}, below)
seems to be  new 
and is a first step towards the understanding of the large time behavior of solutions.

\begin{theorem}[Existence of global-in-time solution]\label{sec2-th0}
Assume that  
\begin{align}
&K^\prime \in L^1 (\R) ,\label{sec2-eq1}\\
&u_0 \in L^1 (\R)
 \cap L^q (\R)\ \text{for
 some}\ 2 \le q < \infty. \label{sec2-eq2}
\end{align}
Suppose also  that $u_0\geq 0$.
Then the initial value problem
\eqref{intro-eq1}--\eqref{intro-eq2} has a unique, nonnegative, global-in-time solution
$u \in C([0,\infty), L^1(\R)\cap L^q(\R))$. 
This solution has the following regularity property
$
u \in C^1((0, \infty), L^p(\R))\cap C((0, \infty), W^{2,p}(\R )))
$
  for each $p\in [1,\infty]$. 
Moreover, the solution
 conserves the integral (``mass'')
\begin{equation}\label{mass}
M\equiv \|u(t)\|_1=\int_\R u(x,t)\;dx = \int_\R u_0(x)\;dx= \|u_0\|_1\quad \mbox{for all} \quad t\geq 0,
\end{equation}
and for each $t_0>0$ and all $p\in (1,\infty)$ we have $\sup_{t>t_0} \|u(t)\|_p<\infty$.
\end{theorem}

In our work \cite{KS09},
we have
proved  that
nonnegative solutions to \eqref{intro-eq1}-\eqref{intro-eq2}   
exist globally-in-time, in the case of ``mildly singular'' kernels satisfying 
$K^\prime \in  L^{q}(\R)$  for some $q\in (1, \infty]$,  see \cite[Thm.~2.5]{KS09}.
Theorem \ref{sec2-th0} improves  those results  in the one dimensional case
by showing the global-in-time existence of nonnegative solutions to \eqref{intro-eq1}--\eqref{intro-eq2} for every  kernel
$K'\in L^1(\R)$. %
This result holds true, in particular, for every ``strongly singular'' kernel \cite{KS09} satisfying
$  K'\in L^1(\R)\setminus L^q(\R)$ for any $q>1$.

Next, we state conditions on $K'$ under which  solutions to \eqref{intro-eq1}-\eqref{intro-eq2} decay as $t\to\infty$.

\begin{theorem}[$L^p$-decay of solutions]\label{sec2-th1}
Assume that $u=u(x,t)$ is a nonnegative solution to problem \eqref{intro-eq1}-\eqref{intro-eq2} with $K'$ and $u_0$ satisfying
 \eqref{sec2-eq1} and \eqref{sec2-eq2}, respectively. There exists
 $D > 0$  and
$ C=C(p, \|K^\prime\|_{1}
 \|u_0\|_{1})>0$  independent of $t$
such that if 
$
\|K^\prime \|_{1} \|u_0\|_{1} \le D,
$
then for each $p \in [1, \infty]$ we have
\begin{align}
\|u(t)\|_p \le C t^{-(1-1/p)/2}  \qquad \mbox{for all}\quad  t > 0.   \label{sec2-eq4}
\end{align}
\end{theorem}

\begin{rem}
In the proof of Theorem \ref{sec2-th1}, we choose $D=1/C_{GNS}$, where $ C_{GNS}$ is the optimal constant 
in  the Gagliardo-Nirenberg-Sobolev inequality  \eqref{sec3-eq2}.
\end{rem}

We do not know if the decay estimate \eqref{sec2-eq4} holds true for every nonnegative solution which does not necessarily
satisfy the condition  
$
\|K^\prime \|_{1} \|u_0\|_{1} \le D.
$
Reasons that such decay estimates may fail for certain kernels $K'$ and for initial conditions with large mass can 
be found below in Theorem \ref{thm:spike} and in the discussion following it.
Now, however, we  prove that  estimates \eqref{sec2-eq4} hold true for each solution   
of problem \eqref{intro-eq1}-\eqref{intro-eq2} which tends to zero as $t\to\infty$ without {\it a priori} assumed decay rate.

\begin{theorem}\label{thm:Lp:decay}
Let the assumptions of Theorem \ref{sec2-th0} hold true. Assume, moreover, that there exists $p_0\in (1,\infty]$ such that 
$\|u(t)\|_{p_0}\to 0$ as $t\to\infty$.
 Then, for each $p \in [1, \infty]$ there is $ C= C(p, \|K^\prime\|_{1},
 \|u_0\|_{1})>0$  independent of $t$ such that 
$
\|u(t)\|_p \le C t^{-(1-1/p)/2}
$
 for all $t > 0$.
\end{theorem}

The main goal of this work is to derive  an asymptotic profile as $t \to \infty$ of those solutions from Theorem \ref{sec2-th0}
which satisfy  the $L^p$-decay estimates \eqref{sec2-eq4}.

\begin{theorem}[Self-similar asymptotics]\label{sec2-th2}
Under the assumptions of Theorem \ref{sec2-th0}, every solution 
$u = u(x, t)$ of problem \eqref{intro-eq1}--\eqref{intro-eq2} 
satisfying  estimate \eqref{sec2-eq4} 
has a self-similar asymptotic profile as $t\to\infty$. More precisely, 
\begin{itemize}
\item[i.] if
$ \int_{\R} K^\prime (x)\, dx = 0$,
  we have 
\begin{align}
t^{(1-1/p)/2} \|u(t) - M \G(t)\|_p \to 0 \quad \text{as}\ t \to
 \infty \label{sec2-eq6}
\end{align}
for every $p \in [1, \infty]$, where $M = \int_{\R} u_0(x)\, dx$ and $\G(x,
 t) = \frac{1}{\sqrt{4 \pi t}} \exp\big(-\frac{|x|^2}{4t}\big)$ is the heat kernel. %

\item[ii.] On the other hand, if
$
A \equiv \int_{\R} K^\prime (x)\, dx \neq 0$, 
we have 
\begin{align}
t^{(1-1/p)/2} \|u(t) - \mathcal{U}_{M, A}(t)\|_p \to 0 \quad \text{as}\ t \to
 \infty \label{sec2-eq8}
\end{align}
for every $p \in [1, \infty]$, where $\mathcal{U}_{M, A}(x, t)=\frac{1}{\sqrt{t}}
\mathcal{U}_{M,A}\big(\frac{x}{\sqrt{t}},1\big)$ 
is the so-called nonlinear diffusion wave and is defined as  the unique self-similar
 solution of the initial value problem for the viscous Burgers equation
\begin{align}
&U_t = U_{xx} - A \left(U^2 \right)_x ,\quad
 \text{for}\ x \in \R,\ t > 0, \label{sec2-eq9}\\
&U(x, 0) = M \delta_0, \label{sec2-eq10}
\end{align}
where $\delta_0$ is the Dirac measure.
\end{itemize}
\end{theorem}
%%%

Let us recall  properties of
solutions to \eqref{sec2-eq9}-\eqref{sec2-eq10}
which will be useful  in the proof of Theorem \ref{sec2-th2}.
It is well-known
 that the Hopf-Cole
transformation  allows us to solve this initial value problem to obtain
the following explicit solution
\begin{equation}\label{nonlin:diff}
\mathcal{U}_{M,A}(x,t)=\frac{At^{-1/2} \exp{(-|x|^2/(4t))}}{C_{M,A} +
 \frac{1}{2}\int_0^{x/\sqrt{t}} \exp{(-\xi^2/4)}\, d\xi},
\end{equation}
where $C_{M,A}$ is a constant which is determined uniquely as a function of
 $M$ and $A$ by the
condition  $\int_\R \U_{M,A}(\eta,1) \;d\eta=M$.
The important point to note here is that
for every $M\in \R$ the function $\U_{M,A}$ is a unique  solution
to equation \eqref{sec2-eq9} in the space  $C((0,\infty);L^1(\R))$ 
having the properties
\begin{equation*}
\int_\R \U_{M,A}(x,t)\;dx =M \quad \mbox{for all}\quad t>0\label{ini1}
\end{equation*}
and
\begin{equation*}
\int_\R \U_{M,A}(x,t) \varphi(x)\;dx \to M\varphi(0) \quad \mbox{as}
\quad  t\to 0\label{ini2}
\end{equation*}
for all $\varphi\in C^\infty_c(\R)$ (\cite[Sec.~4]{EVZ93}).
Such  a solution is called a fundamental solution in the linear theory
and a {\it source solution} in the nonlinear case (cf. \cite{CL, EVZ93, LP}).

In the proof of Theorem \ref{sec2-th2}, we study the behavior, as $\lambda\to\infty$, of 
 the rescaled family of functions
\begin{equation}\label{ul:KL}
 u_\lambda (x, t) = \lambda u(\lambda x, \lambda^2 t) \quad
 \text{and}\quad K^\prime_\lambda (x) = \lambda K^\prime (\lambda x), \quad \mbox{for every}\quad \lambda>0,
\end{equation}
which satisfy the initial value problems
\begin{align}
\partial_t u_\lambda &= \partial_x^2 u_\lambda - \partial_x
 \left(u_\lambda \left(K^\prime_\lambda \ast u_\lambda \right)\right),
 \label{eq:lambda}\\
u_\lambda (x, 0) &= u_{0, \lambda}(x) = \lambda u_0 (\lambda x). \label{ini:lambda}
\end{align}
Notice that 
if $u=u(x,t)$ is a nonnegative solution of problem \eqref{intro-eq1}--\eqref{intro-eq2} obtained in Theorem \ref{sec2-th0},  then 
by \eqref{mass} and by a simple change of variables,
the following identities
\begin{equation}\label{ul:Kl:L1}
\|u_\lambda (t)\|_1 = \|u_0 \|_1\quad \mbox{and}\quad 
 \|K_\lambda^\prime \|_1 = \|K^\prime \|_1
 \end{equation}
hold true for all $t>0$ and all $\lambda>0$.

\medskip
Let us now emphasize that we obtain an asymptotic profile of all solutions to problem 
\eqref{intro-eq1}--\eqref{intro-eq2} which tend to zero as $t\to\infty$.
Additional assumptions on the kernel $K'$ and on  initial conditions
(as those in Theorem~\ref{sec2-th1})  seem to  be  necessary to prove the decay  of solutions
because, in our next theorem, we show a 
concentration phenomenon for some solutions to 
 \eqref{intro-eq1}--\eqref{intro-eq2}.  Here, in order to understand our result, 
one should keep in mind that if a solution $u=u(x,t)$  
behaves for large $t$ either as the heat kernel or the nonlinear diffusion wave, one should expect
that $\|u(t)\|_\infty$ decreases and the first moment $\int_\R u(x,t)|x| \,dx$ increases in $t$.
Here, we find a large class of kernels and initial conditions such that corresponding solutions
have different behavior, at least,  on a certain initial time interval.

\begin{theorem}[Concentration phenomenon]\label{thm:spike}
Assume that the kernel $K'$ satisfies 
\begin{itemize}
\item $K'(x) = K'(|x|)\,{\rm sgn}\,x \;$   for all  $x \in \R\setminus\{0\}$,
\item  $K^\prime (x) \le 0$   for all $x > 0$, 
\item  there exists $\delta > 0$  and
 $\gamma > 0$  such that  $\sup_{0 \le x \le \delta} K^\prime (x)
 \le -\gamma$.
\end{itemize}
Let $u_0\in C_c^\infty (\R)$ be nonnegative, nontrivial, and even.
For every $P>0$, set $u_{0,P}(x)=P^3 u_0(Px)$ and denote by $u_P=u_P(x,t)$ the corresponding solution
of problem \eqref{intro-eq1}--\eqref{intro-eq2} with the kernel $K'$ and $u_{0,P}$ as the initial datum.
If $P>0$ is sufficiently large, then the first moment $I_P(t)=\int_\R u_P(x,t)|x|\,dx $ is a strictly decreasing function of $t\in [0,T]$ for some
$T=T(P)>0$.
\end{theorem}

Some remarks on Theorem \ref{thm:spike} are in order.

\begin{rem}
By the uniqueness, the solution $u_P(x,t)$ considered in Theorem \ref{thm:spike} is an even function of $x$ for every $t>0$.
Hence, it is expected that, for suitable initial conditions, we have  $u_P(0,t)=\max_{x\in\R} u_P(x,t)$. It follows from the proof of 
Theorem \ref{thm:spike} that the quantity 
$u_P(0,t)$ has to increase on an interval $[0,T]$, however, we do not know if it increases monotonically. This phenomenon 
is in perfect agreement with numerical simulations of spikes in the one-dimensional Keller-Segel model, which are reported  in
 \cite{HP04}.
\end{rem}

\begin{rem}
Assumptions on the kernel $K$, as those stated in Theorem \ref{thm:spike}, were imposed in \cite{BKL09,KS09} to show the
finite-time blowup of solutions to aggregation equations either in  the dimension $n\geq 2$ or with an anomalous diffusion modeled by 
the fractional Laplacian.
In this work, however, solutions are global-in-time by Theorem \ref{sec2-th0}, but they have a tendency to form spiky-like structures as those
discussed {\it e.g.} in \cite{H07}.
\end{rem}

\begin{rem}
Notice that $M_P\equiv \int_\R u_{0,P}(x)\,dx =P^2 \int_\R u_{0}(x)\,dx$. Hence, the concentration phenomenon described 
in Theorem \ref{thm:spike} appears only if mass of a solution is sufficiently large. This result should be compared with 
Theorem \ref{sec2-th1}, where solutions are shown to decay for sufficiently small masses.
Moreover, in that case, by the inspection of the proof of Theorem \ref{sec2-th1}, 
one can show that the $L^2$-norm of solutions decays monotonically for all $t>0$. 
\end{rem}

\begin{rem}
It is shown in Theorem \ref{thm:spike} that the first moment $I_P(t)$ is strictly decreasing on a certain initial time interval, however, it cannot
converge to zero when $t\to\infty$. 
Indeed, such a decay of $I_P(t)$ cannot be true because of the inequality (see {\it e.g.} \cite[Remark  2.6]{BK10})
$$
\left(\int_\R u_P(x,t)\,dx\right)^{2-1/p}
\leq C \|u_P(t)\|_p \; I_P(t)^{1-1/p},
$$
the conservation of  mass \eqref{mass}, and the boundedness of the $L^p$-norm of solutions to \eqref{intro-eq1}-\eqref{intro-eq2},
shown in Proposition \ref{sec-ex-prop1}, below.
\end{rem}

\begin{rem}
It is not clear for us what is the large time behavior of solutions to problem \eqref{intro-eq1}-\eqref{intro-eq2} which concentrate initially
in the sense described in Theorem \ref{thm:spike}. Our numerical simulations of solutions to equation \eqref{intro-eq1} on a finite
interval and with the Neumann boundary conditions show their convergence towards nonconstant stationary solutions 
(these results will be published in our subsequent paper).   The large time behavior of solutions to problem  \eqref{intro-eq1}-\eqref{intro-eq2} 
on the whole line $x\in\R$ seems to be more complicated, because one can easily shown that equation \eqref{intro-eq1} 
has no non-zero stationary solutions which decay at infinity sufficiently fast. Indeed,
the equation $w_{xx}-(wK'*w)_x=0$ implies  $w_x-wK'*w =C$, and the constant $C$ has to vanish.
Hence, for $v(x)=\int_{-\infty}^x K'*w(y)\,dy$, we have $(we^{-v})_x=0$ and, consequently, 
$
w(x)=Ce^{v(x)}. 
$
It is easy to check that $w\in L^1(\R)$ and $K'\in L^1(\R)$ 
implies $v\in L^\infty(\R)$, hence, relation $w=Ce^v$  can be true for $C=0$ and $w\equiv 0$, only.
\end{rem}

\begin{rem}
In \cite{NSU03,NY07}, the authors study the large time behavior of solutions to the Cauchy problem for the
so-called parabolic-parabolic model of chemotaxis
\begin{equation}\label{pp-chemo}
u_t=\Delta u -\nabla\cdot (u\nabla v), \quad v_t=\Delta v -v+u, \qquad x\in \R^n, \; t>0.
\end{equation}
Their results can be summarized as follows. If the solution $(u(x, t),
 v(x, t))$ of \eqref{pp-chemo} %
satisfies  $\sup_{t>0} \left(\|u(t)\|_p+\|v(t)\|_p \right) < \infty$ for every $p\in [1,\infty]$,
then $u(x,t)$ decays as a solution to the linear heat equation and
its large time behavior is described by the heat kernel. Moreover, a higher order term of the asymptotic expansion of $u$ is calculated.
Here, however, because of a technical obstacle, we cannot apply methods from \cite{NSU03,NY07} to show a decay of solutions
to \eqref{intro-eq1}-\eqref{intro-eq2}.
\end{rem}
%%%%%%%%%%%%%%%%%%%%%%%%%%%%%%%%%%%%%%%%%%%%%%%%%%%%%%%%%%%%
%%%%%%%%%%%%%%%%%%%%%%%%%%%%%%%%%%%%%%%%%%%%%%%%%%%%%%%%%%%%
%%%%%%%%%%%%%%%%%%%%%%%%%%%%%%%%%%%%%%%%%%%%%%%%%%%%%%%%%%%%

\section{Existence of solutions -- proof of Theorem \ref{sec2-th0}}

The proof of the existence of local-in-time solutions to \eqref{intro-eq1}--\eqref{intro-eq2} is  standard, hence, we
only sketch that reasoning.

%%%%%%%%%%%%%%%%%%%%%%%%%%%%%%%%%%%%%%%%%%%%%%%%%%%%%%

{\it Step 1. Local-in-time solutions.}
We  construct  local-in-time {\it mild} solutions of \eqref{intro-eq1}--\eqref{intro-eq2} which are solutions of  
the following integral equation
\begin{align}
u(t) = \G(\cdot , t)\ast u_0 - \int_0^t \partial_x \G(\cdot , t-s)* \big(u(K' \ast
 u)\big)(s)\, ds \label{duhamel}
\end{align}
with the heat kernel 
$ \G(x, t) = (4\pi t)^{-1/2}\exp\big(-{|x|^2}/(4t)\big)$. 
In our reasoning, we use the following  estimates which 
result immediately from the
Young inequality for the convolution:
\begin{align}
\norm{{\G(\cdot , t)\ast f}}{p} \le C t^{- \frac{1}{2}\left( \frac{1}{q}-\frac{1}{p} \right)}\norm{f}{q}, \label{G1}\\
\norm{{\partial_x \G(\cdot ,t)\ast f}}{p} \le C t^{- \frac{1}{2}
 \left(\frac{1}{q}-\frac{1}{p} \right)- \frac{1}{2}}\norm{f}{q} \label{G2}
\end{align}
for every $1 \le q \le p \le \infty$, each $f \in L^q (\R)$, and   $C=C(p,q)$ independent of $t,f$.
Notice that $C=1$ in inequality \eqref {G1} for $p=q$ because $\|\G(\cdot,t)\|_{L^1}=1$ for all $t>0$.

%5%%

\begin{lemma}[Local existence]\label{local-ex}
Assume that 
$K^\prime \in L^1 (\R)$  and 
$u_0 \in L^1 (\R)
 \cap L^q (\R)$ 
 for some $q \in [2,  \infty)$. %
Then there exists $T=T(\|u_0\|_1, \|u_0\|_q, \|K'\|_1) > 0$ such that the integral equation \eqref{duhamel}
 has the unique solution in the
 space $\mathcal{Y}_T = C([0, T], L^1 (\R)) \cap C([0,
T], L^q (\R))$. 
Moreover, this solution satisfies 
$
 u \in C((0, T], L^p (\R))
 $
for all $1\leq  p \le \infty$.

\end{lemma}
%%%
\begin{proof}
Here, it suffices to follow the reasoning from \cite[Theorem 2.3]{KS09}, where 
local-in-time existence of solutions to the equation \eqref{duhamel}, written as 
$u(t) =
\G(\cdot, t) \ast u_0 + B(u, u)(t)$
with the bilinear form
\begin{align}
B(u, v) (t) = -\int_0^t \partial_x \G(\cdot, t-s)\ast (u(K^\prime \ast
 v))(s)\, ds, \label{sec-ex-eq1}
\end{align}
are constructed in the space $\mathcal{Y}_T$ 
supplemented with the norm $\|u\|_{\mathcal{Y}_T} = \sup_{0
\le t \le T}\|u\|_1 + \sup_{0 \le t \le T}\|u\|_q.$ 
Notice that  $\G(\cdot, t) \ast u_0 \in \mathcal{Y}_T$ by \eqref{G1}.  To apply ideas from 
\cite[Theorem 2.3]{KS09},  one should prove the following estimates of the bilinear form \eqref{sec-ex-eq1}.

First, for every $u, v \in \mathcal{Y}_T$, using \eqref{G1}-\eqref{G2}
 and the Young and the H\"older inequalities, we obtain 
\begin{align*}
\|B(u, v) (t)\|_1 %
&\le C\int_0^t (t-s)^{-1/2} \|u(K^\prime \ast v)(s)\|_1 \, ds \\
&\le C\int_0^t (t-s)^{-1/2} \|u (s)\|_q \|v
 (s)\|_{q^\prime} \|K^\prime \|_1\, ds,
\end{align*}
where  $1/q + 1/q^\prime = 1$. %
Since $1 < q^\prime \le q$ for $q \ge 2$, by a standard interpolation, we have
$
 \|v (s)\|_{q^\prime} \le C (\|v (s)\|_1 + \|v (s)\|_q).
$
Therefore, using the definitions of the norm in $\mathcal{Y}_T $ we obtain 
$
\|B(u, v) (t)\|_1  \le C T^{1/2} \|K^\prime \|_1 \|u\|_{\mathcal{Y}_T}
\|v\|_{\mathcal{Y}_T}.
$

In a similar way, we prove the following $L^q$-estimate
\begin{align}
\|B(u, v) (t)\|_q &\le C\int_0^t (t-s)^{-(1-1/q)/2 - 1/2} \|u(K^\prime
 \ast v)(s)\|_1 \, ds \label{ex-101}\\
&\le C T^{1/2q} \|K^\prime \|_1 \|u\|_{\mathcal{Y}_T} \|v\|_{\mathcal{Y}_T}.\nonumber
\end{align}
Summing up these inequalities, we obtain the following estimate of the bilinear form
\[
 \|B(u, v) \|_{\mathcal{Y}_T} \le C (T^{1/2} + T^{1/2q})\|K^\prime
 \|_1 \|u\|_{\mathcal{Y}_T} \|v\|_{\mathcal{Y}_T}.
\]
Hence,  choosing $T > 0$ such that 
$
4 C (T^{1/2} + T^{1/2q})\|K^\prime
 \|_1 (\|u_0 \|_1 + \|u_0 \|_q) < 1,
$
we obtain the  solution in $\mathcal{Y}_T$ by \cite[Lemma 3.1]{KS09}.

Obviously, by an interpolation inequality, the solution 
of the integral equation \eqref{duhamel} in the space $\mathcal{Y}_T$ 
belongs also to $C([0, T], L^p (\R))$ for every $p\in [1,q]$.
To show that 
$
 u \in C((0, T], L^p (\R))
 $
for all $q < p \le \infty$, it is sufficient to note that $\G(t)*u_0$ belongs to this space by \eqref{G1}. Moreover,
for every $u, v
 \in \mathcal{Y}_T$,
\[
  \int_0^t \partial_x \G(\cdot, t-s)\ast (u(K^\prime \ast
 v))(s)\, ds \in C((0, T], L^p (\R))
\]
for each  $p\in (q, \infty]$
by estimates similar to those in \eqref{ex-101}.
\end{proof}
%%%%%%

%%%%%%%%%%%%%%%%%%%%%%%%%%%%%%%%%%%%%%%%%%%%%%%%%%%%%%
{\it Step 2. Regularity of mild solutions.}
Results on regularity of mild solutions to \eqref{intro-eq1}--\eqref{intro-eq2}  are
standard and well-known, see {\it e.g.}~the monograph by Pazy \cite{pazy}. 
In particular, by a bootstrap argument, one can show that 
any mild solution $u\in C((0,T], L^p(\R))$ obtained in  Lemma~\ref{local-ex} 
satisfies
$u\in C^1((0,T], L^p(\R))\cap C((0,T], W^{2,p}(\R))$.

%%%%%%%%%%%%%%%%%%%%%%%%%%%%%%%%%%%%%%%%%%%%%%%%%%%%%%
{\it Step 3. Positivity and  mass conservation.}
If the initial condition is nonnegative, the same property is shared by the corresponding solution.
Obviously, this fact holds true  for the viscous 
transport equation  $u_t+u_{xx}+(b(x,t)u)_x=0$. Now, it suffices to substitute $b=K'*u$ to show that solutions to 
\eqref{intro-eq1} are nonnegative if  initial conditions are so.

Next, integrating equation \eqref{duhamel} with respect to $x$, using the Fubini theorem, and the identities
$
\int_{\R}\G(x,t)\,dx=1
$
and
$
\int_{\R}\partial_x \G(x,t)\,dx=0
$
for all $t>0$,
we obtain the conservation of the integral and the $L^1$-norm of nonnegative solutions stated in \eqref{mass}.

%%%%%%%%%%%%%%%%%%%%%%%%%%%%%%%%%%%%%%%%%%%%%%%%%%%%%%
{\it Step 4. Uniform $L^p$-estimate.}
In order to show that the solution constructed in Lemma~\ref{local-ex} exists for all $T >
0$, it is sufficient to prove that its $L^q$-norm 
does not blow up in finite time. 
Here, we prove that, in fact, the quantity $\|u(t)\|_p$ is uniformly
bounded for large $t > 0$ for every $p\in [1,\infty)$.

%%%
\begin{prop}\label{sec-ex-prop1}
Assume that $u \in C([0,T], L^1(\R))\cap C([0,T], L^q(\R))$ is a nonnegative local-in-time solution 
of problem
 \eqref{intro-eq1}--\eqref{intro-eq2} for  some $T>0$, with the kernel $K'$, and the initial datum $u_0$ satisfying 
\eqref{sec2-eq1} and
 \eqref{sec2-eq2}, respectively.
For every $p \in [1, \infty)$ and for every $t_0>0$, there
 exists $C = C (p,t_0, \|K^\prime \|_1, \|u_0\|_q , \|u_0 \|_1)$ independent of
 $t$ and of $T>0$ such that 
$
 \|u (t)\|_p \le C
$
for all $t > 0$.
\end{prop}

\begin{proof}
It follows from Lemma \ref{local-ex} that for every $t_0\in (0,T]$ we have $u(t_0)\in L^p(\R)$ for all $p\in [1,\infty]$.
Hence, without loss of generality, we can assume that $u_0\in L^p(\R)$ and $t_0=0$.
Moreover, the regularity of solutions discussed in Step 2 allows us to justify our calculations below.

We proceed by induction to show the uniform bound for the $L^p$-norm with $p = 2^n$, $n \in \N \cap
 \{0\}$. %
The estimates for other $p \in [1, \infty)$ are a simple
 consequence of the H\"older inequality.

For $p = 1$, we have
$
 \|u (t)\|_1 = \|u_0 \|_1
$
because $u$ is nonnegative and because the integral of $u$ is conserved
 in time, see \eqref{mass}. 

Now, let $\varepsilon>0$ be small and we fix it at the end of this proof.
For $p = 2^n$ with $n \ge 1$, we multiply equation
 \eqref{intro-eq1} by $u^{p-1}$ and integrate over $\R$ to obtain
\begin{equation}\label{sec3-eq6}
\begin{split}
\frac{1}{p}\frac{d}{dt}\int_{\R} u^p \, dx %
&= -\frac{4(p-1)}{p^2}\int_{\R} |(u^{p/2})_x|^2 \, dx + (p-1)\int_{\R} u^{p-1}u_x (\widetilde{K}^\prime \ast u)\,
 dx \\
&\quad + (p-1)\int_{\R} u^{p-1}u_x \left((K^\prime - \widetilde{K}^\prime)
 \ast u \right)\, dx, 
\end{split}
\end{equation}
where the auxiliary kernel $\widetilde{K}^\prime \in C^\infty_c
 (\R)$   satisfies   $\|K^\prime - \widetilde{K}^\prime
 \|_1 \leq  \varepsilon$ .

The second term on the right-hand side of \eqref{sec3-eq6} is estimated
 by the $\varepsilon$-Young inequality 
$ab\leq \varepsilon a^2+C(\varepsilon) b^2$
and by  \eqref{mass} as follows
\begin{equation}\label{sec3-eq6-bis}
\begin{split}
(p-1)\Big|\int_{\R} u^{p-1}u_x &(\widetilde{K}^\prime \ast u)\,
 dx \Big| \\
&\le (p-1)\varepsilon \int_{\R} u^{p-2}|u_x|^2 \, dx +
 C(\varepsilon)\|(\widetilde{K}^\prime \ast u)(t)\|_\infty^2 \int_{\R}u^p
 \, dx \\
&\le \frac{4(p-1)}{p^2}\varepsilon \int_{\R} |(u^{p/2})_x|^2 \, dx +
 C(\varepsilon)\|\widetilde{K}^\prime \|_\infty^2  \|u_0\|_1^2  \int_{\R}u^p
 \, dx.
\end{split}
\end{equation}

Concerning the third term on the right-hand side of \eqref{sec3-eq6}, it
 follows from the identity $u^{p-1}u_x=(2/p) (u^{p/2})_xu^{p/2}$, from the H\"older inequality, and from the Young inequality
 that
\begin{equation}\label{sec3-eee}
\begin{split}
(p-1)\Big|\int_{\R} u^{p-1} u_x &(K^\prime - \widetilde{K}^\prime)
 \ast u \, dx \Big| \\
&\le 
\frac{2(p-1)}{p} \|(u^{p/2})_x u^{p/2}\|_{\frac{p+2}{p+1}} \|(K^\prime -
 \widetilde{K}^\prime )\ast u (t)\|_{p+2} \\
&\le \frac{2(p-1)}{p} \|(u^{p/2})_x \|_2 \|u^{p/2}\|_{\frac{2(p+2)}{p}} \|K^\prime -
 \widetilde{K}^\prime \|_1 \|u (t)\|_{p+2} \\
&\le \frac{2(p-1)}{p} \varepsilon \|(u^{p/2})_x \|_2 \|u^{p/2}\|_{\frac{2(p+2)}{p}}^{1+2/p},
\end{split}
\end{equation} 
because $\|K^\prime - \widetilde{K}^\prime \|_1 \leq 
 \varepsilon$. %
Now, notice  that  the  Gagliardo-Nirenberg-Sobolev inequality leads to
\[
 \|u^{p/2}\|_{\frac{2(p+2)}{p}} \le C 
 \|(u^{p/2})_x \|_2^{(p+4)/(3(p+2))}  \|u^{p/2}\|_1^{2(p+1)/(3(p+2))}.
\]
Moreover,
 applying the induction  hypothesis for  $p/2 =
 2^{n-1}$, we have $\|u^{p/2} (t)\|_1=\|u(t)\|_{p/2}^{p/2}\leq C$ 
with $C>0$ independent of $t$ and $T>0$.
Hence, coming back to inequality \eqref{sec3-eee} we obtain
\begin{align}
 (p-1)\left|\int_{\R} u^{p-1} u_x \left((K^\prime - \widetilde{K}^\prime)
 \ast u \right)\, dx \right| \le %
\varepsilon C \|(u^{p/2})_x \|_2^{4(p+1)/3p}\label{sec-ex-eq3}
\end{align}
with $C = C(p, \|\widetilde{K}^\prime \|_\infty, \|u_0 \|_1)$ independent of
 $t$ and  $T$.

Since $\kappa (p)\equiv  4(p+1)/3p \leq 2$ for all $p\geq  2$, using the elementary inequality 
$s^{\kappa(p)}\leq C(s^2+1)$ for all $s\geq 0$ and fixed $C>0$ independent of $s$, we 
deduce from \eqref{sec-ex-eq3} that
\begin{align}
(p-1)\left|\int_{\R} u^{p-1} u_x \left((K^\prime - \widetilde{K}^\prime)
 \ast u \right)\, dx \right| \le %
C \varepsilon \|(u^{p/2})_x \|_2^{2} + \varepsilon C . \label{sec-ex-eq4}
\end{align}
Now,  by  estimates \eqref{sec3-eq6-bis} and \eqref{sec-ex-eq4}, we obtain from
 \eqref{sec3-eq6}  the following inequality
\begin{align}
\frac{d}{dt}\int_{\R} u^p \, dx \le -(1- \varepsilon C )\|(u^{p/2})_x (t)\|_2^2 + C
 \|u (t)\|_p^p + \varepsilon C, \label{sec-ex-eq6}
\end{align}
where  $C = C(p, \|K^\prime\|_1,  \|\widetilde K^\prime\|_\infty , \|u_0\|_1)$ denotes various constants independent of $\varepsilon$,
$u$, $t$, and $T$. 

Now, we require $\varepsilon<C^{-1}$ in the fist term on the right hand side of \eqref{sec-ex-eq6}. 
Using the Nash inequality
\begin{equation}\label{Nash}
  \|v\|_2 \le C \|v_x \|_2^{1/3}\|v\|_1^{2/3},
\end{equation}
which is valid for all $v \in L^1 (\R)$ such that $v_x \in L^2 (\R)$, and
 applying the inductive hypothesis (notice that $p/2 = 2^{n-1}$),  we have
\begin{align}
\|u(t)\|_p^{p/2}=\|u^{p/2} (t)\|_2 &\le C \|\left(u^{p/2} (t)\right)_x \|_2^{1/3}
 \|u^{p/2} (t)\|_1^{2/3} \le C \|\left(u^{p/2}(t)\right)_x \|_2^{1/3},\label{sec-ex-eq5}
\end{align}
where $C $ is independent  of
 $u$, $t$, and $T$.
 Hence,
applying  estimate  \eqref {sec-ex-eq5} in \eqref{sec-ex-eq6}     we obtain the following
 differential inequality for $\|u(t)\|_p^p$:
\[
 \frac{d}{dt}\|u(t)\|_p^p \le -C(1-\varepsilon C) \big(\|u(t)\|_{p}^p \big)^3+ C\|u(t)\|_p^p +
 \varepsilon C,
\]
where $C= C(p, K^\prime , \|u_0\|_1) > 0$ is independent  of $\varepsilon$,
$u$, $t$, and $T$. 

We leave for the reader the proof that any nonnegative solution of the
 differential inequality
$
 f^\prime \le - C (1-\varepsilon C) f^3 + C f + \varepsilon C
$
is bounded, provided  $\varepsilon>0$
 is sufficiently small.  %
Hence, by the recurrence argument, $\|u (t)\|_p^p$ is bounded
 for any $p = 2^n$, $n \in \N \cup \{0\}$ and  this completes the
 proof of  Proposition \ref{sec-ex-prop1}. 
\end{proof}
%%%

\begin{rem}
Notice that we do not control the growth in $p$ of the constants $C$ in Proposition \ref{sec-ex-prop1},
hence we are  not able to pass to the limit  $p \to \infty$  to obtain  the global-in-time
 bound of the $L^\infty$-norm of the solution. 
However, in Theorem \ref{sec2-th1}, we show a  decay estimate  of $\|u (t)\|_\infty $
under additional assumptions on the kernel $K'$.
\end{rem}

\begin{proof}[Proof of Theorem \ref{sec2-th0}.]
Steps 1--4, described above,  contain all details of the proof.
\end{proof}

%%%%%%%%%%%%%%%%%%%%%%%%%%%%%%%%%%%%%%%%%%%%%%%%%%%%%%%%%%%%
%%%%%%%%%%%%%%%%%%%%%%%%%%%%%%%%%%%%%%%%%%%%%%%%%%%%%%%%%%%%
%%%%%%%%%%%%%%%%%%%%%%%%%%%%%%%%%%%%%%%%%%%%%%%%%%%%%%%%%%%%

\section{Optimal $L^p$-decay of solutions}

We are in a position to prove the decay estimates from Theorem \ref{sec2-th1} and we do it
 in two steps. 
First, we obtain the
optimal decay estimate of $L^2$-norm using  Gagliardo-Nirenberg-Sobolev 
inequalities. 
Next, estimates of other $L^p$-norms are shown by applying the
 integral formulation \eqref{duhamel} of the initial value problem \eqref{intro-eq1}-\eqref{intro-eq2}.

%%%%%%%%%%%%%%%%%%%%%%%%%%%%%%%%%%%%%%

\begin{proof}[Proof of Theorem \ref{sec2-th1}]

Let $p=2$.
Multiplying equation \eqref{intro-eq1} by $u$ and integrating the resulting equation  over $\R$
we
 have
\begin{equation}\label{sec3-eq1}
\frac{1}{2}\frac{d}{dt} \int_{\R} u^2 \, dx =  -\int_{\R} |u_x|^2 \, dx
 + \int_{\R} uu_x (K^\prime \ast u)\, dx  
\end{equation}
By the H\"older inequality,  the Young inequality, 
the following Gagliardo-Nirenberg-Sobolev inequality  
\begin{align}
\|u (t) \|_4 \le C_{GNS}  \|u_x (t)\|_2^{1/2} \|u(t)\|_1^{1/2}, \label{sec3-eq2}
\end{align}
and identity \eqref{mass},  the
 second term on the right-hand side of \eqref{sec3-eq1} is estimated as
 follows
 \begin{equation}\label{sec3-eq2-bis}
\begin{split}
\int_{\R} uu_x (K^\prime \ast u)\, dx &\le \|uu_x (t)\|_{4/3}
 \|(K^\prime \ast u) (t)\|_4 \\
&\leq  \|u(t)\|_4^2 \|u_x (t)\|_2 \|K^\prime \|_1 \\
&\leq  C_{GNS} \|u_x (t)\|_2^2 \|K^\prime \|_1 \|u_0\|_1 .
\end{split}
\end{equation}
Coming back to \eqref{sec3-eq1} we see that 
\begin{equation*}
\frac{1}{2}\frac{d}{dt}\|u(t)\|_2^2 \le  - \left(1 - C_{GNS} \|K^\prime \|_1 \|u_0 \|_1 \right)\|u_x (t)\|_2^2, 
\end{equation*}
hence, 
if $ \|K^\prime \|_1 \|u_0 \|_1<1/C_{GNS} $,  we
 obtain
\begin{align}
\frac{d}{dt} \|u(t)\|_2^2 + C \|u_x (t)\|_2^2 \le 0, \label{sec3-eq3}
\end{align}
where $C= C (\|K^\prime\|_1, \|u_0 \|_1) > 0$. 
Now, by the Nash inequality \eqref{Nash}, since the $L^1$-norm of the solution is constant in time by  \eqref{mass},
 we obtain the differential
inequality
\begin{equation}
\frac{d}{dt} \|u(t)\|_2^2 + C \|u_0\|_1^{-4} \big(\|u(t)\|_2^2\big)^3  \le 0, \label{sec3-eq5}
\end{equation}
which implies 
 $\|u(t)\|_2 \le C t^{-1/4}$ for all $t>0$ and $C>0$ independent of $t$.

Now, we are going to use systematically the above $L^2$-decay estimate 
 to show the decay of other $L^p$-norms.  
First, we consider $p \in [1, \infty)$ and we compute the $L^p$-norm of both sides of the integral equation \eqref{duhamel}
to obtain
\begin{equation}\label{Lp:est}
\begin{split}
\|u(t)\|_p &\le \|\G(\cdot, t)\ast u_0 \|_p + \int_0^t \|\partial_x \G(\cdot, t-s)\ast \left(u
 (K^\prime \ast u)\right)(s)\, ds \|_p \, ds \\
&\le C t^{-\frac{1}{2}\left(1-\frac{1}{p}\right)}\|u_0 \|_1 + C \int_0^t
 (t-s)^{-\frac{1}{2}\left(1-\frac{1}{p}\right)-\frac{1}{2}}\|K^\prime
 \|_1 \|u(s)\|_2^2 \, ds
\end{split}
\end{equation}
after applying inequalities \eqref{G1}-\eqref{G2}.
Hence, by the $L^2$-decay estimate, we have
\begin{align*}
\|u(t)\|_p &\le  C t^{-\frac{1}{2}\left(1-\frac{1}{p}\right)}\|u_0 \|_1
 + C(\|K^\prime \|_1 , \|u_0 \|_1) \int_0^t
 (t-s)^{-\frac{1}{2}\left(1-\frac{1}{p}\right)-\frac{1}{2}}s^{-\frac{1}{2}}\,
 ds \\
&= C  t^{-\frac{1}{2}\left(1-\frac{1}{p}\right)} \quad \mbox{for all} \quad t>0,
\end{align*}
because 
$\int_0^t
 (t-s)^{-\frac{1}{2}\left(1-\frac{1}{p}\right)-\frac{1}{2}}s^{-\frac{1}{2}}\,
 ds =
 t^{-\frac{1}{2}\left(1-\frac{1}{p}\right)} \int_0^1
 (1-\rho)^{-\frac{1}{2}\left(1-\frac{1}{p}\right)-\frac{1}{2}}
 \rho^{-\frac{1}{2}}\, d\rho$.

To deal with  the case $p = \infty$, we use the already proved decay estimate  for any $p > 2$. 
Proceeding in a  way similar to that in \eqref{Lp:est}, we obtain 
\begin{align*}
\|u(t)\|_\infty &\le \|\G(\cdot, t) \ast u_0 \|_\infty + \int_0^t
 \|\partial_x \G(\cdot, t-s)\ast (u(K^\prime \ast u))(s)\|_\infty \, ds\\
&\le Ct^{-\frac{1}{2}}\|u_0\|_1 + \int_0^t C
 (t-s)^{-\frac{1}{p}-\frac{1}{2}} \|u(K^\prime \ast u)(s)\|_{p/2}\, ds\\
&\le Ct^{-\frac{1}{2}}\|u_0\|_1 + \int_0^t C
 (t-s)^{-\frac{1}{p}-\frac{1}{2}} \|K^\prime \|_1 \|u (s)\|_p^2 \, ds\\
  &\le C t^{-\frac{1}{2}}\|u_0 \|_1 + C (p, \|K^\prime
 \|_1 , \|u_0\|_1)\int_0^t (t-s)^{-\frac{1}{p}-\frac{1}{2}}
 s^{-\left(1-\frac{1}{p}\right)}\, ds\\
&=C(p, \|K^\prime
 \|_1 , \|u_0\|_1) t^{-\frac{1}{2}}
\end{align*}
for all $t>0$. This completes the proof of 
Theorem \ref{sec2-th1}.
\end{proof}

%%%%%%%%%%%%

In the proof of Theorem \ref{thm:Lp:decay}, we need the following auxiliary result.

\begin{lemma}\label{sec3-new-lem3}
Let $v_x \in L^2 (\R)$. For $\widetilde{K}^\prime \in C^\infty_c
 (\R)$, we define
$
 \widetilde{K} (x) = \int_{-\infty}^x \widetilde{K}^\prime (y)\, dy.
$
Denote
$
 A = \int_{\R} \widetilde{K}^\prime (y)\, dy = \lim_{y \to +\infty}\widetilde{K}(y).
$
Then 
\[
 \|\widetilde{K}^\prime \ast v - Av \|_\infty \le
 \left(\|\widetilde{K}\|_{L^2 (-\infty, 0)} + \|\widetilde{K} -A\|_{L^2
 (0, +\infty)}\right)\|v_x\|_2.
\]
\end{lemma}
%%%
\begin{proof}
First, notice that $\widetilde{K} \in L^2 ((-\infty, 0])$ and
 $\widetilde{K}-A \in L^2 ([0, +\infty))$, because $\widetilde{K}^\prime
 \in C^\infty_c (\R)$. %
Hence, the proof is the immediate consequence of the integration by
 parts and of the Schwartz inequality in view of the following
 identities %
\begin{align*}
\widetilde{K}^\prime * v(x)-Av(x)=
&\int_{\R} \widetilde{K}(y)v_x(x-y)\, dy - A \int_0^{+\infty}v_x(x-y)\,
 dy \\
=& \int_{-\infty}^0 \widetilde{K}(y)v_x(x-y)\, dy +
 \int_0^{+\infty}\left[\widetilde{K}(y) - A \right]v_x (x-y)\, dy.
\end{align*}
\end{proof}
%%%

\begin{proof}[Proof of Theorem  \ref{thm:Lp:decay}]
Without loss of generality, we can assume that $\|u(t)\|_2\to 0$ as $t\to\infty$,
because this is the immediate consequence of the H\"older inequality, Proposition~\ref{sec-ex-prop1},
and the assumption on the decay of the $L^{p_0}$-norm.

To show the optimal decay of the $L^2$-norm, we use the following equality ({\it cf. \eqref{sec3-eq6}})
\begin{equation}\label{sec3-new-eq7}
\begin{split}
 \frac{1}{2}\frac{d}{dt}\|u(t)\|_2^2 =& - \|u_x (t)\|_2^2 
+ \int_{\R}u(t)
 u_x (t)(K^\prime-\widetilde K') \ast u(t)\, dx\\
&+ \int_{\R}u(t)
 u_x (t)\widetilde K^\prime \ast u(t)\, dx,
\end{split}
\end{equation}
where  $\widetilde{K} \in C^\infty_c (\R)$ satisfies 
 $\|K^\prime -
 \widetilde{K}^\prime \|_1 \le \varepsilon$ with $\varepsilon>0$ to be chosen later on.
 
It follows from the H\"older inequality, the Young inequality,  and from the  Gagliardo-Nirenberg-Sobolev inequality 
\eqref{sec3-eq2} (see the calculations which lead to \eqref{sec3-eq2-bis})
that
\begin{equation}\label{A}
\begin{split}
\left|\int_{\R} u(t)u_x(t) (K^\prime - \widetilde{K}^\prime) \ast u(t)\, dx \right|%
&\le C \|u_x (t)\|_2^2 \|K^\prime - \widetilde{K}^\prime \|_1 \|u(t)\|_1 \\
&\leq C \varepsilon \|u_0\|_1\|u_x (t)\|_2^2.
\end{split}
\end{equation}
Next, notice that $\int_{\R}u^2 u_x \, dx = 0$ for all $u\in W^{1,2}(\R)$,
hence, for $A=\int_\R  \widetilde{K}'(y)\,dy$,
 the last term on the right-hand side of \eqref{sec3-new-eq7} is
 estimated as follows
\begin{equation}\label{B}
\begin{split}
\left|\int_{\R}u (t)u_x(t) \widetilde{K}^\prime \ast u(t)\, dx \right| &= %
\left|\int_{\R} u (t)u_x(t) \left(\widetilde{K}^\prime \ast u (t) - A
 u(t)\right)\, dx \right|\\
&\le \|\widetilde{K}^\prime \ast u (t) - A u(t)\|_\infty \|u(t)\|_2
 \|u_x (t)\|_2.
\end{split}
\end{equation}

Now, we apply both estimates \eqref{A} and \eqref{B} 
in equality \eqref{sec3-new-eq7}. Using, moreover,
  Lemma~\ref{sec3-new-lem3} to deal with the last term on the right-hand side of \eqref{B}, 
we deduce  the following
 inequality
\begin{align}
\frac{1}{2}\frac{d}{dt}\|u(t)\|_2^2 \le \left(-1 + C\varepsilon \|u_0
 \|_1 + C\|u(t)\|_2 \left(\|\widetilde{K}\|_{L^2 (-\infty, 0)} +
 \|\widetilde{K} - A\|_{L^2 (0, +\infty)}\right)\right)\|u_x (t)\|_2^2. \label{sec3-new-eq9}
\end{align}
Hence, for sufficiently small $\varepsilon > 0$  and for sufficiently large $T_0 > 0$, 
since $\|u(t)\|_2\to 0$ if $t\to\infty$,
we obtain the estimate
\begin{align}
\frac{1}{2}\frac{d}{dt}\|u(t)\|_2^2 \le -C \|u_x (t)\|_2^2 \label{sec3-new-eq10}
\end{align}
for all $t \ge T_0$ and $C>0$ independent of $t$ and $u$.
Now, it remains to repeat the reasoning from the proof of Theorem \ref{sec2-th1} for $p=2$
(see inequalities \eqref{sec3-eq3}-\eqref{sec3-eq5}) to obtain the required decay of the $L^2$-norm.
To show the decay estimate for other $L^p$-norms, one should copy the corresponding arguments from the proof of Theorem 
\ref{sec2-th1}.
\end{proof}

%%%%%%%%%%%%%%%%%%%%%%%%%%%%%%%%%%%%%%%%%%%%%%%%%%%%%%%%%%%%
%%%%%%%%%%%%%%%%%%%%%%%%%%%%%%%%%%%%%%%%%%%%%%%%%%%%%%%%%%%%
%%%%%%%%%%%%%%%%%%%%%%%%%%%%%%%%%%%%%%%%%%%%%%%%%%%%%%%%%%%%

\section{Self-similar large time behavior}
%%%090907
%%%%
Our goal in this section is to prove Theorem \ref{sec2-th2}. Here, we
always assume that $u = u(x, t)$ is the nonnegative global-in-time
solution of the initial value problem
\eqref{intro-eq1}--\eqref{intro-eq2} with $K$ and $u_0$ satisfying
\eqref{sec2-eq1} and \eqref{sec2-eq2}, respectively. Moreover, we assume that this solution satisfies the following
 decay estimates
\begin{align}
\|u(t)\|_p \le C t^{-\frac{1}{2}\left(1-\frac{1}{p}\right)} 
 \label{sec4-eq1}
\end{align}
for each $p \in [1, \infty]$, all $t > 0$, and $C$ independent of
$t$. 

The proof that the large time behavior of the solution $u = u(x, t)$ is
described either by the
fundamental solution of the heat equation or by the self-similar solution of the viscous Burgers
equation is based on the so-called scaling method which is often use in the study of
asymptotic properties of solutions to nonlinear evolution equation (see {\it e.g.} the review article
\cite{V02}
for some applications of this method to the porous media equation).
Here, for every $\lambda>0$, we denote by  $u_\lambda (x, t) = \lambda u(\lambda x, \lambda^2 t)$  the
solution of the initial value problem
\eqref{eq:lambda}--\eqref{ini:lambda}. 
In the following, we  systematically use identities \eqref{ul:Kl:L1} as well as the decay estimate \eqref{sec4-eq1}.

Now, we prove a series of technical lemmas which 
usually should be obtained to apply the scaling method.
%%%
\begin{lemma}\label{sec4-lem1}
For each $p \in [1, \infty]$ there exists $C=C(\|K^\prime \|_1 , \|u_0
 \|_1)>0$, independent of $t$ and of $\lambda$, such that  
\begin{align}
\|u_\lambda (t)\|_p \le Ct^{-\frac{1}{2}\left(1-\frac{1}{p}\right)} \label{sec4-eq4:ul}
\end{align}
for all $t > 0$ and all $\lambda > 0$.
\end{lemma}
%%%
\begin{proof}
By the change of variables and estimate  \eqref{sec4-eq1} we obtain
\begin{align*}
\|u_\lambda (t)\|_p =\lambda^{1-\frac{1}{p}} \|u(\cdot,
 \lambda^2 t)\|_p\le C \lambda^{1-\frac{1}{p}}\left(\lambda^2 t
 \right)^{-\frac{1}{2}\left(1-\frac{1}{p}\right)} = C t^{-\frac{1}{2}\left(1-\frac{1}{p}\right)}.
\end{align*}
\end{proof}
%%%

\begin{lemma}\label{sec4-lem2}
For each $p \in [1, \infty)$ there exists $C=C(p,\|K^\prime \|_1 , \|u_0
 \|_1)>0$, independent of $t$ and of $\lambda$, such that  
$
\|\partial_x u_\lambda (t)\|_p \le C
 t^{-\frac{1}{2}\left(1-\frac{1}{p}\right)-\frac{1}{2}} 
$
for all $t > 0$ and all $\lambda > 0$.
\end{lemma}
\begin{proof}
Here, we use the the following counterpart of the integral equation \eqref{duhamel}
\begin{align*}
\partial_x u_{\lambda} (t+1) &= \partial_x \G(t) \ast u_\lambda (1) \\
& - \int_0^t
 \partial_x \G(t-s) \ast \Big((u_\lambda)_x \left(K^\prime_\lambda \ast
 u_\lambda \right) + u_\lambda \left(K^\prime_\lambda \ast (u_\lambda)_x \right)
 \Big)(s+1)\, ds
\end{align*}
for all $t>0$. Hence, computing the $L^p$-norm and using \eqref{G1}, \eqref{G2}, \eqref{ul:Kl:L1} we obtain
\begin{align*}
\| \partial_x u_{\lambda} (t+1) \|_p &\le %
C t^{-\frac{1}{2}\left(1-\frac{1}{p}\right)-\frac{1}{2}}\|u_\lambda (1)
 \|_1 \\
&\ + \int_0^t C (t-s)^{-\frac{1}{2}}\Big(\|(u_\lambda)_x \left(K^\prime_\lambda \ast
 u_\lambda \right)(s+1)\|_p + \| u_\lambda \left(K^\prime_\lambda \ast
 (u_\lambda)_x \right) (s+1)\|_p \Big)\, ds \\
&\le C t^{-\frac{1}{2}\left(1-\frac{1}{p}\right)-\frac{1}{2}}\|u_0
 \|_1 \\
&\ + C \|K' \|_1 \int_0^t (t-s)^{-\frac{1}{2}}  \|\partial_x u_\lambda (s +
 1)\|_p \|u_\lambda (s + 1)\|_\infty \, ds.
\end{align*}
Next,  we use inequality \eqref{sec4-eq4:ul} with $p=\infty$:
\begin{align*}
\|u_\lambda (s + 1)\|_\infty \le C(\|K^\prime \|_1 , \|u_0
 \|_1)(s+1)^{-\frac{1}{2}}
\le C(\|K^\prime \|_1 , \|u_0
 \|_1)
\end{align*}
for all $s>0$, consequently, 
\begin{align*}
\| \partial_x u_{\lambda} (t+1) \|_p 
&\le C t^{-\frac{1}{2}\left(1-\frac{1}{p}\right)-\frac{1}{2}}\|u_0 \|_1
+ C  \int_0^t (t-s)^{-\frac{1}{2}}\|\partial_x u_\lambda (s + 1)\|_p \, ds.
\end{align*}
Applying the singular Gronwall lemma (see {\it e.g.} \cite[Ch.~7]{H81} ), we conclude that 
\[
 \| \partial_x u_{\lambda} (t+1) \|_p \le C(p,t, \|K^\prime \|_1 , \|u_0 \|_1),
\]
for all $t>0$, where the right-hand side of this inequality is independent of $\lambda$.
In particular, for $t = 1$, there exists $C = C(\|K^\prime \|_1 ,
 \|u_0 \|_1)$ independent of $\lambda$ such that 
$
 \|\partial_x u_\lambda (2)\|_p \le C 
$
for all $ \lambda > 0$.
Next, using the definition of $u_\lambda$, we obtain
$
  \|\partial_x u_\lambda (2)\|_p = \lambda^{2-\frac{1}{p}}\|\partial_x
 u(\cdot, 2\lambda^2)\|_p.
$
Hence, after substituting $\lambda = \sqrt{t/2}$, we arrive at 
$
\| \partial_x u(\cdot , t) \|_p \le C t^{-\frac{1}{2}\left(1-\frac{1}{p}\right)-\frac{1}{2}}
$
for all $t > 0$.
\end{proof}

The proof of our next lemma relies on a form of Aubin-Simon's compactness result that we recall below.

\begin{theorem}[{\cite[Theorem 5]{Simon}}]\label{simon}
Let $X$, $B$ and $Y$ be Banach spaces satisfying $X \subset B \subset Y$
 with compact embedding $X \subset B$. Assume, for $1 \le p \le
 \infty$ and $T > 0$, that 
\begin{itemize}
\item $F$  is bounded in  $L^p (0, T; X)$, 
\item $\{\partial_t f\,:\, f\in F\}$  is bounded in  $L^p (0, T; Y)$. 
\end{itemize}
Then $F$ is relatively compact in $L^p (0, T; B)$ {\rm (}and in $C (0, T; B)$
 if $p = \infty${\rm )}.
\end{theorem}

%%%
\begin{lemma}[Compactness in $L^1_{loc}(\R)$ ]\label{sec4-lem3}
For every $0 < t_1 < t_2 < \infty$ and every $R > 0$, the set
$
 \{u_\lambda \}_{\lambda > 0} \subseteq C([t_1 , t_2], L^1 ([-R, R]))
$
is relatively compact.
\end{lemma}
\begin{proof}
%From G's notes on 20.07.09.
We apply Theorem \ref{simon} with $p=\infty$, $F = \{u_\lambda \}_{\lambda > 0}$, and
\[
 X = W^{1, 1}([-R, R]), \qquad B = L^1 ([-R, R]), \qquad  Y
 = W^{-1, 1}([-R, R]),
\] 
where $R > 0$ is fixed and arbitrary, and $Y$ is the dual space of $W^{1, 1}_0 ([-R, R])$. %
Obviously,  the embedding $X \subseteq B$ is compact by the Rellich-Kondrashov theorem. 

By Lemmas 
 \ref{sec4-lem1} and \ref{sec4-lem2} with $p = 1$, the sets
$
 \{u_\lambda \}_{\lambda > 0} \subseteq L^\infty \left([t_1 , t_2], L^1([-R, R])\right)
$ 
and 
$
 \{\partial_x u_\lambda \}_{\lambda > 0} \subseteq L^\infty \left([t_1 , t_2], L^1([-R, R])\right)
$ 
are bounded.

To check the second condition of Aubin-Simon's compactness criterion, it is suffices to show that there is a
 positive constant $C$ which independent of $\lambda > 0$ such that 
$
 \sup_{t \in [t_1 , t_2]} \|\partial_t u_\lambda \|_Y \le C.
$
Let us show this estimate by a duality argument. 
For every $\phi \in C^\infty_c \left((-R,
 R)\right)$ and  $t \in [t_1 , t_2]$ we have 
\begin{align*}
\left|\int_\R \partial_t u_\lambda (x, t) \phi (x)\, dx \right| &=
 \left|-\int_\R \partial_x u_\lambda \partial_x \phi \, dx + \int_\R
 (\partial_x \phi) u_\lambda (K^\prime_\lambda \ast u_\lambda )\, dx
 \right|\\
 &\leq 
  \|\phi_x \|_\infty( \|u_\lambda(t)\|_1+ \|K^\prime_\lambda\|_1 \|u_\lambda(t)\|_2^2)\\
&\le \|\phi_x \|_\infty C(t_1, t_2, \|K^\prime \|_1 , \|u_0 \|_1)
\end{align*}
by virtue of Lemma \ref{sec4-lem1}. 
Hence, Lemma \ref{sec4-lem3} is proved.
\end{proof}
%%%

\begin{lemma}[Compactness in $L^1(\R)$ ] \label{sec4-lem4}
For every $0 < t_1 < t_2 < \infty$, the set
$
 \{u_\lambda \}_{\lambda > 0} \subseteq C([t_1 , t_2], L^1 (\R))
$
is relatively compact.
\end{lemma}
\begin{proof}
Let $\psi \in C^\infty (\R)$ be nonnegative and satisfy $\psi (x)=0$ for $|x|<1$ and $\psi (x)=1$ for $|x|>2$.
Put  $\psi_R (x) = \psi (x/R)$ for every $R>0$.
Since $u$ is nonnegative, in view of Lemma \ref{sec4-lem3}, using a standard diagonal argument, it suffices to show
that 
\begin{equation}\label{sec4-claim1}
 \sup_{t \in [t_1 , t_2]}\|u_\lambda (t) \psi_R \|_1 \to 0 \quad
 \text{as}\quad  R \to \infty,\quad \mbox{uniformly in}\quad  \lambda \ge 1.
\end{equation}

Multiplying the both sides of  equation \eqref{eq:lambda} by $\psi_R$ and
 integrating over $\R$ and from $0$ to $t$,  we
 obtain
\begin{align*}
&\int_\R u_\lambda (x, t)\psi_R (x)\, dx - \int_\R u_\lambda (x, 0)\psi_R
 (x)\, dx \\
&\quad = \int_0^t \int_\R
 \partial_{xx} \psi_R (x) u_\lambda (x, s)\, dxds  + \int_0^t \int_\R \partial_x \psi_R
 (x) \big(u_\lambda(x,s) \left(K^\prime_\lambda \ast u_\lambda 
 \right)(x,s)\big)\, dxds.
\end{align*}
Since $\partial_{xx} \psi (x) = \psi''(x/R) / R^2$ and $\partial_x
 \psi_R (x) = \psi'(x/R)/R$, we have
\begin{equation}\label{ul:weak}
\begin{split}
\int_\R u_\lambda (x, t) \psi_R (x)\, dx  \le & \int_\R u_{\lambda, 0} (x)
 \psi_R (x)\, dx + \frac{\|\psi''\|_\infty}{R^2}  \int_0^t \| u_\lambda (s)\|_1\,ds \\
&+   \frac{\|\psi'\|_\infty}{R} \|    K^\prime_\lambda\|_1 \int_0^t \|u_\lambda (s)\|_\infty \|u_\lambda (s)\|_1\,ds\\
 \le &\int_\R u_{\lambda, 0} (x)
 \psi_{R} (x)\, dx \\
&+
 \frac{t}{R^2}\|\psi_{xx}\|_\infty \|u_0\|_1  +  \frac{\|\psi'\|_\infty \|u_0\|_1 \|    K^\prime\|_1}{R} \int_0^t \|u_\lambda
 (s)\|_\infty  \, ds
 \end{split}
\end{equation}
because $\|u_\lambda (s)\|_1=\|u_0\|_1$.

Now, notice that, by the change of variables, we have 
\[
 \int_{|x| > R} u_{\lambda,
 0}(x)\, dx = \int_{|x| > \lambda R} u_0 (x)\, dx \le \int_{|x| > R} u_0
 (x)\, dx
\]
for all  $\lambda \ge 1$. 
Moreover,  it follows from Lemma \ref{sec4-lem1} that  $\|u_\lambda (s)\|_\infty \le C s^{-1/2}$ with $C$ independent of $\lambda$.
Hence, we see for all $t \in [t_1 , t_2]$  that 
\[
 \int_\R u_\lambda (x, t) \psi_R (x)\, dx \le \int_{|x| > R} u_0 (x)\, dx + C_2\left( \frac{t_2}{R^2} + \frac{t_2^{1/2}}{R} \right),
\]
where $C_2 = C_2 (\|\psi_{xx}\|_\infty , \|\psi_{x}\|_\infty ,
 \|K^\prime \|_1 , \|u_0 \|_1)$ is independent of $\lambda>1$. This 
proves our claim \eqref{sec4-claim1} because $u_0\in L^1(\R)$.
\end{proof}
%%%

\begin{lemma}[Initial condition]\label{sec4-lem5}
For every test function $\phi \in C_c^{\infty} (\R)$, there exists $C =
 C(\phi, \|K^\prime \|_1 , \|u_0 \|_1 )$ independent of $\lambda$ such that 
\begin{align}
\left|\int_{\R} u_\lambda (x, t)\phi (x)\, dx - \int_{\R} u_{0,
 \lambda}(x)\phi (x)\, dx \right| \le C\left(t + t^{1/2}\right).
\end{align}
\end{lemma}
%%%

\begin{proof}
Following the calculations which lead to estimates \eqref{ul:weak} with $\psi_R$ replaced by $\phi \in C_c^{\infty} (\R)$
we obtain 
\begin{align*} 
&\left|\int_{\R} u_\lambda (x, t)\phi (x)\, dx - \int_{\R} u_{0,
 \lambda}(x)\phi (x)\, dx \right| \\%
&\quad \le \int_0^t \|\phi_{xx}\|_\infty \|u_\lambda (s)\|_1 \, ds + \int_0^t
 \|\phi_x \|_\infty \|K^\prime_\lambda \|_1 \|u_\lambda (s)\|_1
 \|u_\lambda (s)\|_\infty \, ds \\
&\quad \le  \|\phi_{xx}\|_\infty \|u_0\|_1 t +  C\|\phi_x \|_\infty
 \|K^\prime \|_1 \|u_0 \|_1 \int_0^t  s^{-1/2}\, ds \\
&\quad \le C(t + t^{1/2}),
\end{align*}
where $C > 0$ is independent of $\lambda$.
\end{proof}

Now, we are in a position to prove our main result on the large time behavior of solutions to problem
\eqref{intro-eq1}-\eqref{intro-eq2}.

\begin{proof}[Proof of Theorem \ref{sec2-th2}]
By Lemma \ref{sec4-lem4}, for every  $0 < t_1 < t_2 < \infty$, 
the family $\{u_\lambda \}_{\lambda > 0}$ is
 relatively compact in $C([t_1 , t_2], L^1 (\R))$ for any $0 < t_1 < t_2
 < \infty$. Consequently, there exists a subsequence of
 $\{u_\lambda\}_{\lambda > 0}$ (not relabeled) and a function $\bar{u}
 \in C((0, \infty), L^1 (\R))$ such that 
\begin{align}
u_\lambda \to \bar{u} \quad \text{in}\ C([t_1, t_2], L^1 (\R)) \quad
 \text{as}\ \lambda \to \infty.  \label{sec4-eq10} 
\end{align}
Passing to a subsequence, we can assume that 
\begin{align}
u_\lambda (x, t) \to \bar{u}(x, t) \quad \text{as}\ \lambda \to \infty  \label{sec4-eq11} 
\end{align}
almost everywhere in $(x, t) \in \R \times (0, \infty)$. %

Now, multiplying equation \eqref{eq:lambda} by a test function $\phi \in
 C^\infty_c (\R \times (0, \infty))$ and integrating the resulting equation over $\R \times
 (0, \infty)$, we obtain the identity
\begin{align}\label{weak:lambda}
-\int_\R \int_0^\infty u_\lambda \phi_t \, ds dx &= \int_\R
 \int_0^\infty u_\lambda \phi_{xx} \, ds dx + \int_\R \int_0^\infty
 u_\lambda (K^\prime_\lambda \ast u_\lambda)\phi_x \, ds dx.
\end{align}

Recall that $K'_\lambda (x) =\lambda K'(\lambda x)$ and $A=\int_\R K'(y)\,dy= \int_\R K'_\lambda (y)\,dy$  
for all $\lambda>0$, hence, 
by the well-known property of an approximation of the identity, we have
$K^\prime_\lambda \ast \bar u(t)\to A \bar u(t)$ in $L^1(\R)$ as $\lambda\to\infty$.
Consequently,
by the Young inequality, \eqref{ul:Kl:L1}, and \eqref{sec4-eq10}, 
we have
\begin{align*}
\|K'_\lambda *u_\lambda(t)-A\bar u(t)\|_1\leq
 \|K'\|_1 \|u_\lambda(t)-\bar u(t)\|_1+ \| K'_\lambda * \bar u (t) -A\bar u(t)\|_1 \to 0
\end{align*}
as $\lambda\to \infty$  for every $t>0$.
Therefore, passing to the limit $\lambda \to \infty$  in equality \eqref{weak:lambda} and using the
 properties of the sequence $\{u_\lambda \}_{\lambda > 0}$ stated in
 \eqref{sec4-eq10} and \eqref{sec4-eq11}, we obtain that $\bar{u}(x, t)$
 is a weak solution of the equation
\[
 \bar{u}_t = \bar{u}_{xx} - A (\bar{u}^2)_x
\]
with $A = \int_\R K^\prime (x)\, dx $. %
Now, notice that, by the change of variables and the dominated convergence theorem, we obtain 
$\int_\R u_{0,\lambda}(x)\phi(x)\,dx   = \int_\R u_{0}(x)\phi(x/\lambda)\,dx  \to M\phi(0)$ as $\lambda\to \infty$. 
Hence,
it follows from Lemma \ref{sec4-lem5} that $\bar{u}(x, 0) = M \delta_0$
 in the sense of bounded measures. %
Thus, $\bar{u}$ is a weak solution of the initial value problem
\begin{align}
&\bar{u}_t = \bar{u}_{xx} - A (\bar{u}^2)_x,
 \label{sec4-eq12}\\
&\bar{u}(x, 0) = M \delta_0.  \label{sec4-eq13}
\end{align}
Since problem \eqref{sec4-eq12}-\eqref{sec4-eq13} has a unique solution (see {\it e.g.} \cite[Sec.~4]{EVZ93}), 
the whole family $\{u_\lambda \}_{\lambda > 0}$ converge to
 $\bar{u}$ in $C((0, \infty), L^1 (\R))$. %

Obviously, if $A = 0$, this limit function is the multiple of
 Gauss-Weierstrass kernel
\begin{equation}\label{heat:kernel}
 \bar{u}(x, t) = M\G(x,t)= M \frac{1}{\sqrt{4 \pi t}} \exp\big(-\frac{|x|^2}{4t}\big).
\end{equation}
For $A \neq 0$, we obtain the self-similar solution $\bar{u} =\U_{M,A}$ 
of the viscous
 Burgers equation, given by the explicit formula \eqref{nonlin:diff}.

Hence, by \eqref{sec4-eq10}, we have
\[
 \lim_{\lambda \to \infty} \|u_\lambda (1) - \bar{u}(1)\|_1 = 0
\]
and, after setting $\lambda = \sqrt{t}$ and using the self-similar form of
  $\bar{u}(x, t) = t^{-1/2} \bar{u}
 (xt^{-1/2}, 1)$, we obtain 
\begin{equation} \label{lim:L1}
  \lim_{t \to \infty} \|u (t) - \bar{u}(t)\|_1 = 0.
\end{equation}

The convergence of $u(\cdot, t)$ towards the self-similar profile in the
 $L^p$-norms for $p\in (1, \infty)$ is the immediate consequence of the H\"older inequality,
  the decay estimate \eqref{sec4-eq1} with $p=\infty$, and  \eqref{lim:L1}.
Indeed,  we have 
 \begin{equation} \label{lim:Lp}
 \|u(t)-\bar u(t)\|_p\leq \big(\|u(t)\|_\infty +\|\bar u (t)\|_\infty\big)^{1-1/p} \|u(t)-\bar u(t)\|_1^{1/p}
= o(t^{-(1-1/p)/2})
\end{equation}
  as  $t\to \infty$.

To complete  the
 proof of Theorem \ref{sec2-th2}, it remains to show the convergence in the $L^\infty$-norm. Here, however,
 it suffices notice the decay estimate $\|u_x(t)\|_2\leq Ct^{-3/4}$, provided by Lemma \ref{sec4-lem2} with $\lambda=1$,
 and the identity $\|\bar u_x(t)\|_2= t^{-3/4}\|\bar u_x(1)\|_2$ resulting  from the explicit formulas \eqref{heat:kernel} and \eqref{nonlin:diff}.
 Hence, by the Gagliardo-Nirenberg-Sobolev inequality and by \eqref{lim:Lp} with $p=2$, we obtain
 \begin{equation*}
 \|u(t)-\bar u(t)\|_\infty \leq C \big(\|u_x(t)\|_2 +\|\bar u_x (t)\|_2\big)^{1/2} \|u(t)-\bar u(t)\|_2^{1/2}
= o(t^{-1/2})
 \end{equation*}
 as  $t\to \infty$. 
\end{proof}

%%%%%%%%%%%%%%%%%%%%%%%%%%%%%%%%%%%%%%%%%%%%%%%%%%%%%%%%%%%%
%%%%%%%%%%%%%%%%%%%%%%%%%%%%%%%%%%%%%%%%%%%%%%%%%%%%%%%%%%%%
%%%%%%%%%%%%%%%%%%%%%%%%%%%%%%%%%%%%%%%%%%%%%%%%%%%%%%%%%%%%

\section{Concentration phenomenon}

\begin{proof}[Proof of Theorem \ref{thm:spike}]
First, notice that equation equation \eqref{intro-eq1} is invariant under the transformation $x\mapsto -x$ because, by the assumptions,
the kernel $K'$ is an odd function.
Hence, by the uniqueness of solutions to problem \eqref{intro-eq1}-\eqref{intro-eq2}, the solution $u_P$ corresponding to
the even and nonnegative initial datum $u_{0,P} $
satisfies $u_P(x, t) = u_P(-x, t)\geq 0$ for all $x\in \R$ and $t>0$. 

We study the evolution
of the first moment
$
 I_P(t) = \int_{\R}u_P(x, t)|x|\, dx
$
following ideas from the recent papers 
\cite{BKL09,KS09}. 
First, notice that integrating by parts we have
\[
 \int_{\R}|x| \partial_x^2 u_{P}(x, t)\, dx =
2 \int_{0}^\infty  x \partial_x^2 u_{P}(x, t)\, dx 
= - 2 \int_0^\infty \partial_x u_{P}(x, t)\, dx =
 2 u_P(0, t),
\]
because $\partial_xu_P(0,t)=0$ for all $t>0$ in the case of the even function $u_P(\cdot, t)$.
 Next, 
by the assumptions, we have  $K^\prime
(x) = \frac{x}{|x|}K^\prime (|x|)$ where
$ \frac{x}{|x|} = \text{sign}\ x$. 
 Hence, using equation \eqref{intro-eq1}, we obtain
\begin{equation}
\label{I:eq}
\begin{split}
\displaystyle\frac{d}{dt} I_P(t) &= \int_{\R} \partial_t u_{P} (x, t)|x|\, dx \\
&= \int_{\R} \left(\partial^2_x u_{P}(x,t) - \partial_x(u_P(x,t)K^\prime \ast u_P(x,t) )\right)|x|\, dx \\
&= 2u_P(0, t) + \int_{\R}\int_{\R}u_P(x, t)u_P(y, t) \frac{x}{|x|}
 \frac{x-y}{|x-y|}K^\prime (|x-y|)\, dxdy \\
&= 2u_P(0, t) +
 \frac{1}{2}\int_{\R}\int_{\R}u_P(x,t)u_P(y,t)\left(\frac{x}{|x|} -
 \frac{y}{|y|}\right)\frac{x-y}{|x-y|}K^\prime (|x-y|)\, dxdy
\end{split}
\end{equation}
by the symmetrization of the double integral on the right-hand side.
Now,
we use the elementary identity
\[
 \left(\frac{x}{|x|} - \frac{y}{|y|}\right)\frac{x-y}{|x-y|} =
 \left(\frac{|x| + |y|}{|x-y|}\right)\left(1-\frac{x}{|x|}\cdot \frac{y}{|y|}\right)
\]
and the inequalities
\[
 1 \le \frac{|x| + |y|}{|x-y|}\quad \text{and}\quad 1-\frac{x}{|x|}\cdot
 \frac{y}{|y|} \ge 0
\]
which are valid for all $x,y\in \R\setminus \{0\}$.
Moreover,  using  the properties of $K'$ stated in Theorem~\ref{thm:spike}, we deduce from  \eqref{I:eq} that 
\begin{equation}\label{spike-eq01}
\begin{split}
\displaystyle\frac{d}{dt}I_P(t) &\le 2u_P(0, t) +
 \frac{1}{2} \iint\limits_{|x| < \delta/2 \atop |y| < \delta/2} u_P(x,t)u_P(y,t)\left(1-\frac{x}{|x|}\cdot
 \frac{y}{|y|}\right)K^\prime (|x-y|)\, dxdy \\
&\le 2u_P(0, t) -
 \frac{\gamma}{2} \iint\limits_{|x| < \delta/2 \atop |y| < \delta/2} u_P(x,t)u_P(y,t)\left(1-\frac{x}{|x|}\cdot
 \frac{y}{|y|}\right)\, dxdy \\
&\le 2u_P(0, t) - \frac{\gamma}{2}\int_{\R}\int_{\R} u_P(x,t)u_P(y,t)\left(1-\frac{x}{|x|}\cdot
 \frac{y}{|y|}\right)\, dxdy \\
&\quad + \gamma \int_{\R}\int_{|y| > \delta/2}  u_P(x,t)u_P(y,t)\left(1-\frac{x}{|x|}\cdot
 \frac{y}{|y|}\right)\, dxdy.
\end{split}
\end{equation}
Note now that 
\[
 \int_{\R} u_P(x, t)\frac{x}{|x|}\, dx = 0 \quad \text{and}\quad \int_{|x|
 \ge \delta/2} u_P(x, t) \frac{x}{|x|}\, dx = 0
\] 
because $u_P(x, t) = u_P(-x, t)$.  
Hence, denoting $M_P= \int_\R u_P(x,t)\;dx$ it follows from \eqref{spike-eq01} that 
\begin{equation}\label{spike-eq2}
\begin{split}
\displaystyle\frac{d}{dt}I_P(t) &\le 2u_P(0, t) - \frac{\gamma}{2}M_P^2 +
 \gamma M_P \int_{|y| > \delta/2}u_P(y, t)\, dy \\
&\le 2u_P(0, t) - \frac{\gamma}{2}M_P^2 +
  \frac{2\gamma }{\delta} M_P I_P(t).
\end{split}
\end{equation}
Now, we use the dependence of the initial datum on $P$. Putting  
$$ M= \int_\R u_0(x)\,dx, \quad 
 I(0)= \int_\R u_0(x)|x|\,dx,$$
 and 
changing the variables
we have
$$
2u_{0,P}(0) - \frac{\gamma}{2}M_P^2 +
  \frac{2\gamma }{\delta} M_P I_P(0)
  =
   2u_0(0)P^3
 -  \frac{\gamma}{2}M^2 P^4 +
   \frac{2\gamma }{\delta} M I(0)P^3<0
$$
if $P>0$ is sufficiently large.
Hence, for such $P$, the right hand side of inequality \eqref{spike-eq2} is negative for $t=0$ .
Consequently
by the continuity of the functions $u_P$ and $I_P$,  
the right hand side of inequality \eqref{spike-eq2} is  negative for every $t\in [0, T]$ with some $T=T(P)>0$. 
This completes the proof of Theorem 
\ref{thm:spike}.
\end{proof}

\end{document}